\documentclass[12pt]{amsart} 

\usepackage{a4wide}
\usepackage{amsmath, amssymb, amsfonts, amsthm}

\DeclareMathOperator{\dist}{dist}
\DeclareMathOperator{\BMO}{BMO}

\newcommand{\R}{\mathbb{R}}

\newtheorem{theorem}{Theorem}
\newtheorem{corollary}{Corollary}[theorem]

\title[Duality in a stability problem]{Duality in a stability problem for some functionals arising in interpolation theory}

\author{Anton Tselishchev}

\thanks{This research was supported by the Russian Science Foundation (grant No.~14-21-00035).}

\address{Chebyshev Laboratory, St. Petersburg State University, 14th Line V.O., 29B, Saint Petersburg 199178 Russia}
\address{St. Petersburg Department of Steklov Mathematical Institute, Fontanka 27, St. Petersburg 191023, Russia}
\email{celis-anton@yandex.ru}

\date{}

\begin{document}
\begin{abstract}
By using duality, it is shown that there exist near-minimizers for the distance functionals for the couple $(L^\infty, L^p)$, $1<p<\infty$, that are stable under the action of singular integral operators.
\end{abstract}
\maketitle

\section{Introduction}
Consider the spaces $L^1$ and $L^p$ and the following expression, the distance functional from the function from $L^1$ to the ball of radius $s$ in $L^p$:
$$
 E(s,f;L^1,L^p)=\dist_{L^1}(f, B_{L^p}(s))=\inf\{\|f-g\|_{L^1}: \|g\|_{L^p} \leq s\}.
$$

The book \cite{KK}, among other things, solves the question of the existence of a stable (under the action of a singular integral
  operator) near-minimizer for such functional. Specifically, the following theorem is proved there.
  \begin{theorem}
Let $T$ be a Calder\' on --Zygmund operator and $f\in L^1$ is a function for which $Tf\in L^1$. Then for any $s>0$ there exists
such function $u^{(s)}\in L^1$ that the following conditions hold:
\begin{align*}
   \|u^{(s)}\|_{L^p}\lesssim& s,\\
   \|f-u^{(s)}\|_{L^1}\lesssim& \dist_{L^1}(f,B_{L^p}(s)),\\
   \|Tf-Tu^{(s)}\|_{L^1}\lesssim& \dist_{L^1}(f,B_{L^p}(s))+\dist_{L^1}(Tf,B_{L^p}(s)).
  \end{align*}
\end{theorem}

Here we say that $A\lesssim B$ if $A\leq CB$ for some constant $C$. It will always be clear from the context from which parameters $C$ can depend and from which it can not (or it will be stated explicitly). Here these constants do not depend on $s$ and $f$.

The first two conditions in this theorem mean that $u^{(s)}$ is a near-minimizer for the distance functional for $f$ at $s$ and the third one says that $Tu^{(s)}$ behaves much like the near-minimizer for the distance functional for $Tf$ at $s$ (in particular, it will be the near-minimizer if the second term majorizes the first one). The main method of proof
is the approach of Bourgain from the paper \cite{Bo}. Stable near-minimizer is constructed almost
explicitly --- srecifically, an arbitrary near-minimizer turns into a stable one by adding the “good” part of
  Calder\' on--Zygmund decomposition
  
  In addition to the application of stability theorems to the interpolation theory, it is not difficult to prove, for example, the following corollary of the above theorem.
  (which is also proved in the book \cite{KK}).
  \begin{corollary}
  If $T$ and $f$ are as above, then there exist functions $f_k$ in $L^1\cap L^p$ convergent to $f$ in $L^1$ and such that $Tf_k$ are all in $L^1$ and $\|Tf_k-Tf\|_{L^1}\rightarrow 0$. 
  \end{corollary}
  
By the same method, but using some other decompositions instead of the standard Calderon--Zygmund, one can get
  similar theorems in some other cases that are not considered in the book \cite{KK} --- in particular, the stability theorem
  with respect to projections on wavelets with only a rather weak decay condition at infinity ---
  this is done by the author in the paper \cite{Tsel}.  
  
Singular integral operators are usually discontinuous not only on $L^1$ but also on $L^\infty$. Thus the question of existence of stable under the actions of such operators near-minimizers for couple $(L^\infty, L^p)$, $1<p<\infty$, arises. This problem is not solved in the book \cite{KK} and moreover it is hard to expect that such near-minimizers can be constructed explicitly in any sense. The goal of this article is to show by using duality that such near-minimizers do exist.

The author is kindly greatful to his scientific advisor, S. V. Kislyakov, for posing these problems and for the continuous support during the process of their solutions.

\section{The application of duality}
We note that the problem of the existence of stable near-minimizers is connected to another, more classical --- the problem of $K$-closedness of a certain pair of subspaces. The definition of the notion of $K$-closedness introduced in the paper \cite{Pis} is as follows.
Let $(X_0, X_1)$ be a compatible pair of Banach spaces (which means that they are embedded into some topological vector space) and $Y_0$ and $Y_1$ are closed
subspaces of $X_0$ and $X_1$ respectively. Then this pair of subspaces is called $K$-closed in $(X_0, X_1)$ if there exists a constant $C$ such that for any representation of element
$y \in Y_0 + Y_1$ in the form $y = x_0 + x_1$, where $x_0 \in X_0$, $x_1 \in X_1$, we can find another representation $y = y_0 + y_1$, where
$y_0$ and $y_1$ are in $Y_0$ and $Y_1$ respectively, and we can control their norms: $\|y_0\|_{Y_0} \leq C \|x_0\|_{X_0}$,
$\|y_1\|_{Y_1}\leq C \|x_1\|_{X_1}$. The concept of $K$-closedness plays an important role in interpolation theory --- if one knows
the interpolation space for the pair $(X_0, X_1)$ (denoted by the symbol
$(X_0, X_1)_{\theta, q}$ for some $\theta$ and $q$),
and a pair of subspaces $(Y_0, Y_1)$ is $K$-closed in it, then the corresponding
interpolation space for the pair $(Y_0, Y_1)$ is easy to determine --- the following equality is true:
$$
(Y_0, Y_1)_{\theta,q}=(X_0, X_1)_{\theta, q}\cap(Y_0+Y_1).
$$

Let $T$ be a singular integral operator (or a projection related to wavelets from the paper \cite{Tsel},
in this case, $L^p$ stands for $L^p(\R)$ and instead of
Calder\' on--Zigmund decomposition one should simply use the decomposition described in that paper). For any finite $p$
we denote by $X_p$  the space $L^p \oplus L^p$, and $Y_p=\{(f, Tf):f\in L^p, Tf\in L^p\}$ denotes its subspace.
Clearly, $Y_p$ is a closed subspace of $X_p$ (for $p>1$ this is obvious,
since $T$ is a bounded operator on $L^p$, and
for $p=1$ it is also easy in view of the fact that $T$ is an operator of weak type $(1,1)$).
It is known that the question about $K$-closedness of couple $(Y_1, Y_p)$ in $(X_1, X_p)$ (for example, in one-dimensional case, if one sets $T$ to be the Riesz projection then this question transforms into the problem of $K$-closedness of couple of Hardy spaces $(H^1, H^p)$ in the couple $(L^1, L^p)$) is solved positively, see for example \cite{K} or \cite{KX}.

For an element $(u, Tu)\in Y_1+Y_p$ we write:
$$
(u,Tu)=(u_0, v_0)+(u_1, v_1)\in X_1+X_p.
$$

When we investigate the question of $K$-closedness, we are interested in norms of summands in the right hand side of the equation in $L^1\oplus L^1$ and $L^p\oplus L^p$ respectively, so we will assume that  $\|u_0\|_{L^1}\leq a$, $\|v_0\|_{L^1}\leq a$, $\|u_1\|_{L^p}\leq b$, 
$\|v_1\|_{L^p}\leq b$. In this case from the $K$-closedness of the couple $(Y_1, Y_p)$ in $(X_1, X_p)$ we get the following decomposition:
$$
(u,Tu)=(v, Tv) + (w, Tw),
$$
where $\|v\|_{L^1}\lesssim a$, $\|Tv\|_{L^1}\lesssim a$, $\|w\|_{L^p}\lesssim b$, $\|Tw\|_{L^p}\lesssim b$.

We now come back to the problem of the existence of a stable near-minimizer for the couple $(L^1, L^p)$. It can also be rewritten ih the similar terms. Indeed, if we fix a positive number $b$, then we can take $u_0$ and $v_0$ such that their norms in
$L^1$ does not exceed $2\dist_{L^1}(f,B_{L^p}(b))$ and $2\dist_{L^1}(Tf,B_{L^p}(b))$, respectively
(and $u_1$ and $v_1$ can be taken from a ball of radius $b$ in $L^p$). Thus, we write:
$$
(u, Tu)=(u_0, v_0) + (u_1,v_1),
$$
where $\|u_0\|_{L^1}\leq a$, $\|u_1\|_{L^p}\leq b$, $\|v_1\|_{L^p}\leq b$, $\|v_0\|_{L^1}\leq c$. Our goal is to construct a decomposition of the form $(u, Tu)=(\alpha, T\alpha)+(\beta, T\beta)$ where the norms of the summands are controlled by the same numbers $a$, $b$ and $c$. Note that, in essence, this is exactly what is done when applying the Bourgain method in the proof of the theorem 1 in the book \cite{KK} (as well as in the proof of Theorem 2 in \cite{Tsel}) but for the sake of completeness we repeat this argument here.
So, if $u_0=g+h$ is the Calder\' on--Zygmund decomposition on the level $\lambda$ ($g$ is the "good" part, $h$ is "bad"),
where $\lambda$ is such that $\lambda^{p-1}a=b^p$, then $(u, Tu)=(h, Th)+(u_1+g, T(u_1+g))$. Besides that, 
$\|h\|_{L^1}\lesssim a$ and $\|u_1+g\|_{L^p} \leq b +\|g\|_{L^p}\lesssim b$ because $|g|\lesssim\lambda$ and $\|g\|_{L^1}\lesssim a$ (and so $\|g\|_{L^p}=(\int |g|^p)^{1/p}\lesssim (\lambda^{p-1}a)^{1/p}=b$). Since $T$ is bounded on $L^p$, the inequality 
$\|T(u_1+g)\|_{L^p}\lesssim b$ is also true. Let us denote the cubes from the Calder\' on--Zygmund decomposition by $\{Q_i\}$ and let $\Omega$ be $\cup 10Q_i$. Then, using the properties of the operator $T$ ("long-range $L^1$ regularity", using the terminology of book \cite{KK}), $\|Th\|_{L^1(\Omega^c)} \lesssim a$ where $\Omega^c$ denotes the complement of set $\Omega$. In order to estimate the integral over $\Omega$, we use the H\" older inequality in the following way:
$$
\|Th\|_{L^1(\Omega)}=\|v_0+v_1-T(u_1+g)\|_{L^1(\Omega)}\lesssim c + |\Omega|^{1/p'}b.
$$
Measure of set $\Omega$ can be estimated by $\frac{\|u_0\|_{L^1}}{\lambda}\leq\frac{a}{\lambda}$. According to our choise of 
$\lambda$ we get the inequality $\|Th\|_{L^1(\Omega)}\lesssim a+c$. Thus, once we have the decomposition of pair $(u, Tu)\in X_1+X_p$ stated above, we get the following decomposition:
$$
(u,Tu)=(h, Th)+(u_1+g, T(u_1+g))\in Y_1+Y_p,
$$
where $\|h\|_{L^1}\lesssim a$, $\|u_1+g\|_{L^p}\lesssim b$, $\|T(u_1+g)\|_{L^p}\lesssim b$,
$\|Th\|_{L^1}\lesssim a+c$.

This discussion shows that the stability problem is more delicate than the problem about $K$-closedness stated above --- some part of information is not used in the $K$-closedness problem.

For a subspace $Y$ of a Banach the space $X$ we denote by $Y^\perp$ the annihilator of $Y$ --- such elements $x^*\in X^*$ that
$\langle y, x^*\rangle =0$ for all $y\in Y$. It is clear that $Y^\perp$ is a weak-* closed subspace of $X^*$. In the paper \cite{Pis} it is noted that the $K$-closedness of the couple of subspaces $(W_0, W_1)$ in $(Z_0, Z_1)$ is equivalent to $K$-closedness of the couple of their annihilators, $(W_0^\perp, W_1^\perp)$, in $(Z_0^*, Z_1^*)$. The exposition of this fact can be found in the paper
\cite{K}. For the problem of near-minimizers, a similar argument can be made and we pass to it.

It is not difficult to realise what is the annihilator of the subspace $Y_p$. Let $(\alpha,\beta)\in Y_p^\perp$, then
$\langle (f,Tf), (\alpha,\beta)\rangle=0$ for all $f$, which can be rewritten as $\langle f,\alpha+T^*\beta\rangle=0$ and
thus $Y_p^\perp$ has the form $\{ (-T^* \beta, \beta) : \beta \in L^{p'}\}$. What we have just written is true only in
the case $1<p<\infty$, when $T^*$ is a bounded operator on $L^p$. The case, for example, $p'=\infty$, should be treated with
caution --- in particular, then $T^*\beta$ is an element of the space $\BMO$, in which functions are defined only up to the constant. This (again, for the $K$-closedness question) is written, for example, in the paper \cite{KX}. According to this discussion,
in the formulation of the theorem we will assume that the original function $ f $
lies in $L^\infty\cap L^q$ for some $q<\infty$ --- then the function
$T^* f$ is uniquely defined (as a function in $L^q$). Note that if $f\in L^\infty\cap L^q$ and $T^*f\in L^\infty$, then
$(-T^*f,f)\in Y_1^\perp$. Indeed, according to Corollary 1.1, for every function $g\in L^1 $, such that $Tg \in L^1$,
there exists a sequence of functions $g_n \in L^1 \cap L^{q'}$ such that $Tg_n\in L^1$, $g_n\rightarrow g$ and
$Tg_n\rightarrow Tg$ (both convergences are in $L^1$). Therefore it is enough for
$\langle(-T^*f,f), (g,Tg)\rangle$ to be equal to zero for all $g\in L^1\cap L^{q'}$, and this is
obviously true because $f\in L^q$, and
$T$ is bounded on $L^q$ and $L^{q'}$.

After all the remarks we made, it is not difficult to prove the theorem about the existence of a stable near-minimizer for the pair $(L^\infty, L^p)$.

\begin{theorem}
Suppose that $1<q<\infty$, f is a function from $L^\infty\cap L^q$ such that $T^*f\in L^\infty$ and $1<p<\infty$. Then for any $s>0$ there exists a function $v^{(s)}\in L^p$ for which the following conditions hold:
\begin{align*}
  \|v^{(s)}\|_{L^p}\lesssim& s,\\
  \|f-v^{(s)}\|_{L^\infty}\lesssim& \dist_{L^\infty}(f,B_{L^p}(s)),\\
  \|T^*f-T^*v^{(s)}\|_{L^\infty}\lesssim& \dist_{L^\infty}(f, B_{L^p}(s))+\dist_{L^\infty}(T^*f, B_{L^p}(s)).
 \end{align*}
\end{theorem}

\begin{proof}
Suppose $(-T^*f,f)$ is decomposed in the following way:
$$
 (-T^*f,f)=(\alpha_0, \alpha_1)+(\beta_0,\beta_1),
$$
where $\|\alpha_0\|_{L^\infty}\leq t$, $\|\alpha_1\|_{L^\infty}\leq r$, $\|\beta_0\|_{L^p}\leq s$, 
 $\|\beta_1\|_{L^p}\leq s$. We are going to show that in this case $(-T^*f,f)\in U_\infty+U_{p}$ where
 \begin{align*}
  &U_{\infty}=\{(v_0, v_1)\in Y_1^\perp:\|v_0\|_{L^\infty}\leq c(t+r), \|v_1\|_{L^\infty}\leq cr\},\\
  &U_p=\{(w_0, w_1)\in L^p\oplus L^p: \|w_0\|_{L^p}\leq cs, \|w_1\|_{L^p}\leq cs, w_0=-T^*w_1\}\subset Y_{p'}^\perp.
 \end{align*}
 Here the constant $c$ will be chosen later. It is clear that the statement of the theorem follows from this --- as in the beginning of this section, it is enough to take $r=2\dist_{L^\infty}(f,B_{L^p}(s))$ and $t=2\dist_{L^\infty}(T^*f, B_{L^p}(s))$ (and to denote by $\beta_0$ and $\beta_1$ the corresponding near-minimizers).
 
So we suppose it is not true. It is easy to see that $U_\infty+U_p$ is a weakly-* compact
 subset of $L^\infty\oplus L^\infty+L^p\oplus L^p$. Therefore, if the point $(-T^*f,f)$
does not lie in the set $U_\infty+U_p$, then they can be separated by a weakly-* continuous functional on
$L^\infty\oplus L^\infty+L^p\oplus L^p$, that is, we can find the element $F\in L^1\oplus L^1\cap L^{p'}\oplus L^{p'}$ such that
$$\langle (-T^*f, f), F\rangle>1$$ and $$\sup \{|\langle x, F\rangle|:x\in U_{\infty}+U_p\}<1.$$ The second of these equations
gives an estimate for the norm of the functional  $F$ restricted to $Y_1^\perp$ and to $Y_{p'}^\perp$. The functional on
$Y_1^\perp$ can be considered as an element of the factor space $L^1\oplus L^1 / Y_1$, and it can be "lifted" to the element of
$L^1\oplus L^1$, with the norm increasing by no more than twice. By doing the same for $L^{p'}\oplus L^{p'}$, we can get the pairs of
functions $(\phi_0, \phi_1)\in L^1\oplus L^1$ and $(\psi_0, \psi_1)\in L^{p'}\oplus L^{p'}$ such that
$\|\phi_0\|_{L^1}\leq\frac{2}{c(t+r)}$, $\|\phi_1\|_{L^1}\leq\frac{2}{cr}$, $\|\psi_0\|_{L^{p'}}\leq\frac{2}{cs}$,
$\|\psi_1\|_{L^{p'}}\leq\frac{2}{cs}$, and the actions of these pairs as functionals on $Y_1^\perp$ and $Y_{p'}^\perp$ respectively
coincide with the action of $F$. This means, in particular, that $(\phi_0, \phi_1)-(\psi_0,\psi_1)$ annihilates
$Y_1^\perp\cap Y_{p'}^\perp$, and therefore lies in $Y_1+Y_{p'}$. According to what is written at the beginning of this section, it means that this difference can be
written in the following form:
$$
(\phi_0, \phi_1)-(\psi_0,\psi_1)=(v_0, v_1)-(w_0,w_1),
$$
where $(v_0,v_1)\in Y_1$, $(w_0, w_1)\in Y_p$, $\|v_0\|_{L^1}\leq\frac{C}{c}\frac{2}{t+r}$, $\|v_1\|_{L^1}\leq
\frac{C}{c}(\frac{2}{t+r}+\frac{2}{r})$, $\|w_0\|_{L^{p'}}\leq \frac{C}{c}\frac{2}{s}$, 
$\|w_1\|_{L^{p'}}\leq \frac{C}{c}\frac{2}{s}$. Here $C$ is a constant hiding under the sign "$\lesssim$" from the reasoning
at the beginning of this section (or, equivalently, in Theorem 1). We set $G=(\phi_0, \phi_1)-(v_0,v_1)=(\psi_0, \psi_1)-(w_0, w_1)$.
Then, since $(v_0, v_1)\in Y_1$, $G$ coincides with $F$ on $Y_1^\perp$ (as a functional), and since
$(w_0,w_1)\in Y_{p'}$, $G$ coincides with $F$ on $Y_{p'}^\perp$. Since $(-T^*f, f)\in Y_1^\perp+Y_{p'}^\perp$, we write:
$$
\langle (-T^*f,f), F\rangle=\langle (\alpha_0, \alpha_1), -(v_0,v_1)+(\phi_0, \phi_1)\rangle +
\langle (\beta_0, \beta_1), -(w_0, w_1)+(\psi_0, \psi_1)\rangle.
$$
The right hand side is estimated by Hölder's inequality --- it does not exceed
$$
\frac{2(C+1)}{c}\Big(\frac{t}{t+r}+r\Big(\frac{1}{t+r}+\frac{1}{r}\Big)+2\Big)=\frac{6(C+1)}{c}.
$$
Taking $c=10(C + 1)$, we arrive at a contradiction with the fact that 
$$ 
\langle (-T^*f, f), F\rangle>1,
$$ and the theorem is proved.
\end{proof}

In the book \cite{KK} (and the paper \cite{Tsel}) it is noted that Theorem 1 is also true
for the operator $\chi_E T$, where $E$ is an arbitrary measurable
subset of
$\R^d$, if we are considering singular integrals, and a subset of $\R$, if we are considering wavelets. In this regard,
the theorem we just proved can be formulated for the operator $T^*\chi_E$ --- that is, for the function $f$ living on the set $E$
(and continued by zero to the whole space $\R^d$).

\begin{theorem}
Suppose that $1<q<\infty$, f is a function from $L^\infty\cap L^q$ whose support lies in a measurable set $E$ such that 
$T^*f\in L^\infty$ and $1<p<\infty$. Then for any $s>0$ there exists a function $v^{(s)}\in L^p$ with a support lying inside the set 
$E$ for which the following conditions hold:
\begin{align*}
  \|v^{(s)}\|_{L^p}\lesssim& s,\\
  \|f-v^{(s)}\|_{L^\infty}\lesssim& \dist_{L^\infty}(f,B_{L^p}(s)),\\
  \|T^*f-T^*v^{(s)}\|_{L^\infty}\lesssim& \dist_{L^\infty}(f, B_{L^p}(s))+\dist_{L^\infty}(T^*f, B_{L^p}(s)).
 \end{align*}
\end{theorem}

As we mentioned above, this theorem is exactly the theorem 2 for the operator $T^*\chi_E$ which is true because theorem 1 is true for the operator $\chi_E T$. We note that if $E$ is a set of finite measure, then the condition $f\in L^q$ is redundant because the bounded function on $E$ automatically lies in all $L^q$ for any $q$.


\begin{thebibliography}{99}
\bibitem{Bo} J. Bourgain, \emph{Some consequences of Pisier's approach to interpolation}, Isr. Math. J., 77 (1992), 165--185.
\bibitem{K} S. V. Kisliakov, \emph{Interpolation of $H^p$-spaces: some recent developments}, Israel Math. Conf. Proc. 13 
 (1999), 102--140.
\bibitem{KK} S. Kislyakov and N. Kruglyak, \emph{Extremal Problems in Interpolation Theory, Whitney--Besicovitch Coverings, 
 and Singular Integrals}, Birkh\" auser, 2013.
\bibitem{KX} S. V. Kislyakov, Quan Hua Xu, \emph{Real interpolation and singular integrals}, Algebra i Analiz, \textbf{8}:4 (1996), 75--109
\bibitem{Pis} G. Pisier, \emph{Interpolation between $H^p$ spaces and non-commutative generalizations. I}, Pacific J. 
 Math., \textbf{155} (1992), no. 2, 341--368.
 \bibitem{Tsel} A. Tselishchev, \emph{Stability of nearly optimal decompositions in Fourier Analysis}, Zap. Nauchn. Sem. POMI, \textbf{456} (2018), 191--207.
\end{thebibliography}
\end{document}